\documentclass[11pt]{amsart}

\usepackage[lmargin=1in,rmargin=1in,tmargin=1in,bmargin=1in]{geometry}
\usepackage{amsthm,amsfonts,amssymb,amsmath}
\usepackage{graphicx}
\usepackage{tikz}
\usepackage{tikz-cd}
\usepackage{mathtools,leftindex}
\usetikzlibrary{positioning,arrows}
\usepackage{enumitem}
\usepackage{subcaption}
\usepackage{thmtools}
\usepackage{thm-restate}
\usepackage[
  colorlinks=true,
  linkcolor=blue,
  citecolor=blue,
  hypertexnames=false
]{hyperref}

\usetikzlibrary{calc}

\newcommand{\bZ}{\mathbb{Z}}

\newcommand{\cT}{\mathcal{T}}

\newcommand{\tv}{\tilde{v}}

\newcommand{\hv}{\widehat{v}}
\newcommand{\hu}{\widehat{u}}
\newcommand{\he}{\widehat{e}}

\newcommand{\hX}{\widehat{X}}
\newcommand{\hY}{\widehat{Y}}
\newcommand{\hG}{\widehat{G}}
\newcommand{\hV}{\widehat{V}}
\newcommand{\hgamma}{\widehat{\gamma}}
\newcommand{\hGamma}{\widehat{\Gamma}}
\newcommand{\hTheta}{\widehat{\Theta}}

\newcommand{\hPhi}{\widehat{\Phi}}

\newcommand{\hPsi}{\widehat{\Psi}}

\newcommand{\para}{\parallel}
\newcommand{\npara}{\nparallel}

\newcommand{\dr}{\overrightarrow}
\newcommand{\dl}{\overleftarrow}

\newcommand{\RNum}[1]{\uppercase\expandafter{\romannumeral #1\relax}}

\newcommand{\CAT}{\operatorname{CAT}}

\newcommand{\Aut}{\operatorname{Aut}}

\newcommand{\lk}{\operatorname{lk}}

\newcommand{\SL}{\operatorname{SL}}

\newcommand{\stab}{\operatorname{Stab}}

\newcommand{\ima}{\operatorname{im}}

\newtheorem{thm}{Theorem}[section]

\newtheorem{cor}[thm]{Corollary}
\newtheorem{lem}[thm]{Lemma}

\theoremstyle{definition}
\newtheorem{defn}[thm]{Definition}
\newtheorem{rmk}[thm]{Remark}

\numberwithin{equation}{section}

\begin{document}

\title{Virtual specialness of the double}
\author{Changqian Li}

\address{Department of Mathematics, The Ohio State University, Columbus, OH, 43210, U.S.}
\email{li.13132@osu.edu}

\maketitle

\begin{abstract}
Let $G$ be a virtually compact special Gromov-hyperbolic group. 
We prove that the double $G *_H G$ along a quasiconvex subgroup $H$ is virtually compact special. 
More generally, we show that if a finite graph of groups has constant vertex groups, with each vertex group virtually compact special Gromov-hyperbolic and each edge group quasiconvex in its adjacent vertex groups, then its fundamental group is virtually compact special.
\end{abstract}

\section{Introduction}
A group is \emph{(compact) special} if it has a finite index subgroup that is the fundamental group of a (compact) cube complex satisfying certain combinatorial hyperplane conditions (see Definition~\ref{def of special}). A group is \emph{virtually (compact) special} if it admits a finite index (compact) special group.
The theory of virtually special groups was introduced by Haglund and Wise in \cite{Haglund-Wise2008}. Since then, it has led to many important developments, among which the most notable is Agol's proof of the Virtual Haken Conjecture \cite{Agol2013VirtualHaken}.

Virtually compact special groups satisfy strong algebraic properties. In particular, since compact special groups embed in finitely generated right-angled Artin groups and hence in $\SL_n(\bZ)$, virtually compact special groups are linear and residually finite. 

In \cite{doubaTsouvalas2026}, Douba and Tsouvalas investigated the linearity
of amalgams of subgroups of algebraic groups. The following result from their
paper is particularly relevant to the present work.

\begin{thm}[cf. {\cite[Theorem 1.4]{doubaTsouvalas2026}}]
Let $G$ be a virtually compact special Gromov-hyperbolic group, and let
$H < G$ be a quasiconvex subgroup. Then the double $G *_H G$ is linear.
\end{thm}

In the same paper, Douba and Tsouvalas suggest that the double $G *_H G$
should remain virtually compact special, which would give an alternative proof
of linearity. Our main result confirms this expectation.

\begin{thm}\label{special of double}
Let $G$ be a virtually compact special Gromov-hyperbolic group, and let $H<G$ be a quasiconvex subgroup. Then the double $G*_HG$ is virtually compact special.
\end{thm}

In general, the quasiconvexity assumption on $H$ cannot simply be omitted. 
For example, let $H$ be a subgroup of a virtually compact special Gromov-hyperbolic group $G$ such that $H$ is not separable in $G$ (see \cite[Example~10.3]{Haglund-Wise2008} for an example of construction). 
Then the double $G *_H G$ is not residually finite, and hence is not virtually special. 
Even if $H$ is separable but not quasiconvex in $G$, the double $G *_H G$ may fail to be virtually compact special. However, it is not clear whether it must still be virtually special.

More generally, we say that a graph of groups has \emph{constant vertex groups} if there exists a fixed group $V$ such that all vertex groups are isomorphic to $V$, and for every edge, the two monomorphisms from the edge group to the adjacent vertex groups are compatible with isomorphisms (see Definition \ref{def of constant}). In particular, the double $G *_H G$ is the fundamental group of a graph of groups with constant vertex groups, whose underlying graph consists of a single edge joining two vertices.

Theorem \ref{special of double} is a special case of the following more general result.
\begin{thm}\label{general result}
Let $G$ be the fundamental group of a finite graph of groups with constant vertex groups. Suppose that the vertex groups are virtually compact special Gromov-hyperbolic, and that each edge group is quasiconvex in the adjacent vertex groups. Then $G$ is virtually compact special.
\end{thm}

The theorem fails in general if we do not require the compatibility of the monomorphisms from edge groups in the definition of constant vertex groups. 
For example, the graph of graphs $X$ constructed in \cite[Part \RNum{2}, Section 2]{Wise1996NPCSquareComplexes} (see also \cite[Section 4]{Wise2007CompleteSquareComplexes}) has a single vertex graph.
Hence $\pi_1X$ splits as a graph of free groups with a single vertex group. However, $\pi_1X$ is not residually finite \cite[Corollary 1.4]{CapraceWesolek2018Indicability}, and therefore is not virtually special.

Theorem \ref{general result} is a consequence of the more general Corollary \ref{cor of Gromov-hyperbolic}, which concerns graphs of groups with locally constant vertex groups. (See Definition \ref{def of locally constant} for the definition of locally constant vertex groups and locally constant vertex spaces.) 

More generally, all of the preceding results ultimately follow from the following theorem, which does not assume hyperbolicity of the vertex groups and is stated in the setting of graphs of cube complexes (see Definition~\ref{def: graph of cube complexes}).
\begin{restatable}{thm}{mainthm}
\label{cubical thm}
Let $X$ be a compact connected cube complex that splits as a graph of virtually special cube complexes with locally constant vertex spaces. Then $X$ is virtually special.
\end{restatable}
Theorem \ref{cubical thm} is related to the Malnormal Combination Theorem \cite[Theorem~8.2]{haglund-wise2012} in the following sense. The latter applies in the hyperbolic setting and requires a malnormality assumption on the edge groups, whereas Theorem \ref{cubical thm} does not impose such a hypothesis. Instead, it applies in the more symmetric setting where all vertex spaces are isometric.

\subsection*{Structure of the paper}
We first review the necessary background on $\CAT(0)$ cube complexes and special cube complexes in Sections~\ref{section: cat(0)} and~\ref{section: special cube complex}. 
In Section~\ref{section: graph of cube complexes}, we recall the notation for graphs of cube complexes.
In Sections \ref{section: constant} and \ref{section: locally constant}, we introduce the notions of (locally) constant vertex groups and (locally) constant vertex spaces.
Then we prove our main theorems in Section \ref{section: virtually special}.
Finally, in Section~\ref{section: finite stature}, we make some remarks about the related notion of finite stature.

\subsection*{Acknowledgements}
The author would like to thank Jingyin Huang for the encouragement and many helpful discussions.
The author would also like to thank Sami Douba and Konstantinos Tsouvalas for helpful communication. 
This work was motivated by their work~\cite{doubaTsouvalas2026}.

\section{$\CAT(0)$ cube complexes}\label{section: cat(0)}
We review the following notions of $\CAT(0)$ cube complexes from \cite{wise2021}.
In this paper, all cube complexes are assumed to be connected. 
\begin{defn}
An \emph{$n$-cube} is a copy of $[0,1]^n$. Its subcubes are obtained by restricting some coordinates to $0$ and some to $1$. A \emph{cube complex} is obtained by gluing a collection of cubes along their subcubes isometrically. A \emph{flag complex} is a simplicial complex with the property that every finite set of pairwise adjacent vertices spans a simplex. A cube complex $X$ is \emph{nonpositively curved} if the link $\lk(v)$ of each vertex $v$ of $X$ is a flag complex. A nonpositively curved cube complex is \emph{$\CAT(0)$} if it is simply-connected.
\end{defn}

\begin{defn}
Let $Q$ be an $n$-cube of a cube complex $X$.
A \emph{midcube} of $Q\cong [0,1]^n$ is the subspace obtained by restricting one coordinate to $1/2$. Let $M$ be the disjoint union of the collection of midcubes of cubes of $X$. Let $D$ be the quotient space of $M$ induced by identifying faces of midcubes under the inclusion map. The connected components of~$D$ are called the \emph{hyperplanes} of $X$.
\end{defn}

\begin{defn}
Let $f: Y\to X$ be a combinatorial map of nonpositively curved cube complexes. We say $f$ is a \emph{local isometry} if for each vertex $v$ of $Y$, the induced map of links $\lk(v)\to \lk(f(v))$ is an embedding and has full image (a subcomplex $A$ of a simplicial complex $B$ is \emph{full} if each simplex of $B$ whose vertices are in $A$ is entirely contained in $A$).
\end{defn}

\begin{defn}
Let $f: Y\to X$ be a local isometry of connected nonpositively curved cube complexes and let $p:\hX\to X$ be a covering map. The \emph{fiber product} $Y\times_X\hX$ is defined to be
\[
Y\times_X\hX = \{(a,b)\in Y\times \hX \mid f(a) = p(b)\}.
\]
An \emph{elevation} $\hY$ of $Y$ to $\hX$ is a connected component of $Y\times_X \hX$. 
\end{defn}

\begin{rmk}\label{remark of elevation and subgroup}
Each elevation of $Y$ to $\hX$ is covering of $Y$. 
By covering space theory, there is a one-to-one correspondence between the set of double cosets $(f_*\pi_1Y)\backslash \pi_1X/(p_*\pi_1\hX)$ and connected components of $Y\times_X\hX$, where the component corresponding to $(f_*\pi_1Y)g(p_*\pi_1\hX)$ has fundamental group $g^{-1}(f_*\pi_1Y)g\cap (p_*\pi_1\hX)$.
In particular, when $f_*: \pi_1Y\to \pi_1X$ is surjective and $\hX$ is connected, then $Y\times_X\hX$ is connected, and it is the covering of $Y$ corresponding to $(f_*\pi_1Y)\cap (p_*\pi_1\hX)$.
\end{rmk}

\section{Special cube complexes}\label{section: special cube complex}
We review the following from \cite{Haglund-Wise2008}.
\begin{defn}
Let $X$ be a cube complex. If $e$ is an edge of $X$, then $\dr{e}$ denotes the edge $e$ equipped with a chosen direction, and $\dl{e}$ denotes the same edge with the opposite direction.

Let $Y\subset X$ be a subcomplex.
Two edges $a$ and $b$ of $Y$ are \emph{elementary parallel} if there exists a square $s: [0,1]^2\to Y$ of $Y$ such that $a,b$ are opposite sides of $s$ in the sense that $a = s([0,1]\times \{0\})$ and $b = s([0,1]\times \{1\})$. Two directed edges $\dr{a}$ and $\dr{b}$ are \emph{elementary parallel} if $a,b$ are elementary parallel in some square $s$ of $Y$ such that $\dr{a}$ and $\dr{b}$ are both directed from $s(0,t)$ to $s(1,t)$ for $t = 0$ and $1$ respectively. The relation of elementary parallelism generates an equivalence relation among edges of $Y$, called \emph{parallelism}, on both indirected and directed edges. We write $a\para_Y b$ (resp. $\dr{a}\para_Y \dr{b}$) if $a,b$ (resp. $\dr{a}, \dr{b}$) are parallel in $Y$. When the subcomplex $Y$ is $X$, we simply write $a\para b$ (resp. $\dr{a}\para \dr{b}$).
\end{defn}
\begin{defn}
Let $Y\subset X$ a subcomplex of a cube complex $X$ and let $H$ be a hyperplane of $Y$. An edge $e$ of $Y$ is \emph{dual} to $H$ if the midcube of $e$ is a $0$-cube of $H$. A directed edge $\dr{e}$ is \emph{dual} to $H$ if $e$ is dual to $H$.
\end{defn}
\begin{rmk}
Note that $a \para b$ if and only if $a$ and $b$ are dual to the same hyperplane. However, it is possible that $\dr{a} \npara \dr{b}$ even when $\dr{a}$ and $\dr{b}$ are dual to the same hyperplane, in which case we have $\dr{a} \para \dl{b}$.
\end{rmk}
\begin{defn}\label{def of cross and osculate}
Let $X$ be a cube complex. We introduce the following terminology for hyperplanes of $X$.
\begin{enumerate}
    \item A hyperplane $H$ is \emph{$2$-sided} if there does not exist a directed edge $\dr{e}$ dual to $H$ such that $\dr{e}\para \dl{e}$.
    \item Two hyperplanes $H_1$ and $H_2$ \emph{cross} if there exist a vertex $v$ and edges $e_1,e_2$ incident to $v$ that are consecutive in a square of $X$, where $e_i$ is dual to $H_i$. In this case, we say that $H_1$ and $H_2$ cross at $(v;e_1,e_2)$. When $H = H_1 = H_2$, we say that $H$ crosses itself.
    \item Two hyperplanes $H_1$ and $H_2$ \emph{osculate} if there exist a vertex $v$ and edges $e_1,e_2$ incident to $v$ that are not consecutive in any square of $X$, where $e_i$ is dual to $H_i$. In this case, we say that $H_1$ and $H_2$ inter-osculate at $(v;e_1,e_2)$. When $H = H_1 = H_2$, we say that $H$ self-osculates.
    \item A hyperplane $H$ \emph{directly self-osculates} if there exist a vertex $v$ and distinct directed edges $\dr{e_1},\dr{e_2}$ with initial vertex $v$ such that $H$ self-osculates at $(v;e_1,e_2)$. In this case, we say that $H$ directly self-osculates at $(v;\dr{e_1},\dr{e_2})$.
    \item Two distinct hyperplanes \emph{inter-osculate} if they both cross and osculate.
\end{enumerate}
\end{defn}

\begin{defn}\label{def of special}
A nonpositively curved cube complex $X$ is \emph{special} if:
\begin{enumerate}
    \item no hyperplane crosses itself;
    \item each hyperplane is $2$-sided;
    \item no hyperplane directly self-osculates;
    \item no two hyperplanes inter-osculate.
\end{enumerate}
A cube complex is virtually special if it admits a finite-sheeted cover that is special.
A group is \emph{(compact) special} if it is the fundamental group of a (compact) special cube complex. A group is \emph{virtually (compact) special} if it has a finite index subgroup that is (compact) special.
\end{defn}
\begin{rmk}
\begin{enumerate}
    \item The definition of self-osculation above does not involve directed edges and therefore includes both direct and indirect self-osculations. In some literature, the term self-osculation is used to mean direct self-osculation. For example, see \cite{HaglundWise2010coxeter}.
    \item If we further forbid self-osculation, then the cube complex satisfies the stronger notion of being \emph{directly special}, which appears in \cite{Huang2018Commensurability}, \cite{Shepherd2023}, and \cite{shepherd2025productseparabilityspecialcube}. Among compact cube complexes, the notions of special and directly special are equivalent up to finite covers \cite[Proposition 3.10]{Haglund-Wise2008}.
\end{enumerate}
\end{rmk}

We also introduce the following generalization of crossing and osculation between a hyperplane and a subcomplex, which appeared in \cite[Remark A.9]{OregonReyes2023} and \cite[Definition 2.11]{Shepherd2023}.
\begin{defn}
Let $X$ be a nonpositively curved cube copmlex, and let $Y\subset X$ be a subcomplex. Suppose $H$ is a hyperplane of $X$.
\begin{enumerate}
    \item The hyperplane $H$ \emph{crosses} the subcomplex $Y$ if $H\cap Y\neq\emptyset$.
    \item The hyperplane $H$ and the subcomplex $Y$ \emph{osculate} at $(v; e)$ if $v$ is a vertex of $Y$ and $e$ is an edge of $X \setminus Y$ incident to $v$ and dual to $H$.
    \item The hyperplane $H$ and the sbucomplex $Y$ \emph{inter-osculate} if they cross and osculate.
\end{enumerate}
\end{defn}

The following was proved in \cite[Corollary 4.8]{Shepherd2023} using the walker and imitator construction.
\begin{lem}\label{lemma of interosculation}
Let $Y_1,\cdots, Y_n\to X$ be local isometries of compact virtually special cube complexes. Then there exists a compact special regular cover $\hX\to X$ such that all elevations of the $Y_i$ to $\hX$ are embedded and do not inter-osculate with hyperplanes of $\hX$.
\end{lem}

\section{Graph of cube complexes}\label{section: graph of cube complexes}
We review the notion of graph of group and graph of spaces from \cite{Serre1980} and \cite{ScottWall1979}.
\begin{defn}\label{def: graph of cube complexes}
A \emph{graph of nonpositively curved cube complexes} $X_\Gamma$ over a connected underlying graph $\Gamma$ consists of
\begin{enumerate}
    \item a connected vertex space $X_v$ that is a nonpositively curved cube complex for each vertex $v\in V(\Gamma)$ of $\Gamma$;
    \item a connected edge space $X_e$ that is a nonpositively curved cube complex for each edge $e\in E(\Gamma)$ of $\Gamma$;
    \item locally isometric attaching maps $\phi_e^-: X_e\to X_{\iota(e)}$ and $\phi_e^+: X_e\to X_{\tau(e)}$ for each edge $e$ and its two endpoints $\iota(e),\tau(e)$.
\end{enumerate}
The \emph{total space} of $X_\Gamma$ is defined to be the quotient space 
\[
\left(\bigsqcup_{v\in V(\Gamma)}X_v\right)\sqcup \left(\bigsqcup_{e\in E(\Gamma)}X_e\times[0,1]\right)\Big / \sim
\]
where $X_e\times\{0\}\sim \varphi_e^-(X_e)$ and $X_e\times \{1\}\sim \varphi_e^+(X_e)$.
A cube complex $X$ splits as a graph of nonpositively curved cube complexes $X_\Gamma$ if $X$ is the total space of $X_\Gamma$.
\end{defn}

\begin{defn}
A \emph{graph of groups} $G_\Gamma$ over a connected underlying graph $\Gamma$ consists of
\begin{enumerate}
    \item a vertex groups $G_v$ for each vertex $v\in V(\Gamma)$ of $\Gamma$;
    \item an edge group $G_e$ for each edge $e\in E(\Gamma)$ of $\Gamma$;
    \item monomorphisms $\Phi_e^-: G_e\to G_{\iota(e)}$ and $\Phi_e^+: G_e\to G_{\tau(e)}$ for each edge $e$ and its two endpoints $\iota(e),\tau(e)$.
\end{enumerate}
Let $T$ be a spanning tree of $\Gamma$, then the \emph{fundamental group} of $G_\Gamma$ is the group generated by $\{G_v\mid v\in V(\Gamma)\}$ and $\{t_e\mid e\in E(\Gamma)\}$ subject to the relations
\[
t_e\Phi_e^-(x)t_e^{-1} = \Phi_e^+(x),\quad \forall x\in G_e, \ e\in E(\Gamma)
\quad \text{and}\quad
t_e = 1,\quad \forall e\in E(T).
\]
A groups $G$ splits as a graph of groups if it is the fundamental group of a graph of groups~$G_\Gamma$.
\end{defn}

\begin{defn}
Let $X_\Gamma$ be a graph of nonpositively curved cube complexes. Then for each thickened edge space $X_e\times [0,1]$, we have the following two types of edges:
\begin{itemize}
    \item \emph{Horizontal edges} of the form $x\times [0,1]$ for some vertex $x\in V(X_e)$;
    \item \emph{Vertical edges} of the form $f\times t$ for some edge $f\in E(X_e)$ and $t = 0,1$.
\end{itemize}
A hyperplane of $X_e\times [0,1]$ which is dual to a horizontal edge (resp. vertical edge) is called a \emph{vertical hyperplane} (resp. \emph{horizontal hyperplane}) of $X_e\times [0,1]$. In particular, each thickened edge space $X_e\times [0,1]$ has a unique vertical hyperplane, which is isometric to $X_e$.
\end{defn}
\begin{lem}\label{lemma of vertical hyperplane}
Let $X$ be a cube complex which splits as a graph of nonpositively curved cube complexes $X_\Gamma$. Then for each edge $e$ of $\Gamma$, the following hold:
\begin{enumerate}
    \item The vertical hyperplane of the thickened edge space $X_e\times [0,1]$ does not cross itself, and is $2$-sided.
    \item If the attaching maps $X_e\to X_{\iota(e)}$ and $X_e\to X_{\tau(e)}$ are embedded, then the vertical hyperplane of $X_e\times [0,1]$ does not directly self-osculate.
    \item If the images of attaching maps $X_e\to X_{\iota(e)}$ and $X_e\to X_{\tau(e)}$ do not inter-osculate with hyperplanes of $X_{\iota(e)}$ and $X_{\tau(e)}$, then the vertical hyperplane of $X_e\times [0,1]$ does not inter-osculate with hyperplanes of $X_{\iota(e)}$ and $X_{\tau(e)}$.
\end{enumerate}
\begin{proof}
Let $H$ be the vertical hyperplane of the thickened edge space $X_e\times [0,1]$. Then $H$ is dual only to horizontal edges of $X_e\times [0,1]$. Since horizontal edges are not consecutive in any square of $X$, the vertical hyperplane $H$ does not cross itself.

Furthermore, let $\{x\}\times [0,1]$ and $\{y\}\times [0,1]$ be two horizontal edges that are elementarily parallel (i.e. they are opposite sides of a square). Then the vertices $x$ and $y$ of $X_e$ are joined by an edge $f$. The square containing $\{x\}\times [0,1]$ and $\{y\}\times [0,1]$ has boundary edges
\[
\{x\}\times [0,1],\quad f\times \{1\},\quad \{y\}\times [0,1],\quad f\times \{0\}.
\]
Thus, the horizontal edges $\{x\}\times [0,1]$ and $\{y\}\times [0,1]$ are elementarily parallel as directed edges if they are both directed from $0$ to $1$, or both directed from $1$ to $0$.
Consequently, any two parallel directed horizontal edges in $X_e\times [0,1]$ are directed either both from $0$ to $1$ or both from $1$ to $0$. In particular, the vertical hyperplane $H$ is $2$-sided. This proves~(1).

Suppose that the vertical hyperplane $H$ directly self-osculate at $(v; \dr{e_1}, \dr{e_2})$ for two horizontal edges $e_1 = \{x\}\times [0,1], e_2 = \{y\}\times [0,1]$. If they are both directed from~$0$ to~$1$ with the common initial vertex $v = (x,0) = (y,0)$, then the vertices $x,y$ of $X_e$ have the same image under the attaching map $\phi_e^-: X_e\to X_{\iota(e)}$. In particular, $\phi_e^-$ is not an embedding as shown in Figure \ref{fig: EmbeddingInterosculation}. Similarly, if $e_1,e_2$ are both directed from $1$ to $0$ with the common initial vertex $v = (x,1) = (y,1)$, then the attaching map $\phi_e^+: X_e\to X_{\tau(e)}$ is not an embedding. This proves~(2).

Finally, suppose that the vertical hyperplane $H$ of $X_e\times [0,1]$ inter-osculates with a hyperplane $H'$ of $X_{\iota(e)}$. Then there exist two horizontal edges $e_1 = \{x\}\times [0,1], e_2 = \{y\}\times [0,1]$ of $X_e\times [0,1]$ and two edges $f_1, f_2$ dual to $H'$, such that $H, H'$ cross at $(v_1;e_1,f_1)$ and osculate at $(v_2;e_2,f_2)$. In particular, $e_1,f_1$ are consecutive in a square, and $e_2,f_2$ are not consecutive in any square. Since $e_1$ is a horizontal edge of $X_e\times [0,1]$, $f_1$ is the image of a vertical edge under $\phi_e^-$. Thus, $H'\cap \phi_e^-(X_e)$ contains the midpoint of $f_1$ and in particular is nonempty. As $e_2,f_2$ are not consecutive in any square, $f_2$ is an edge of $X_{\iota(e)}\setminus \phi_e^-(X_e)$. Thus, $H'$ and $\phi_e^-(X_e)$ osculate at $(v_2; f_2)$ as shown in Figure~\ref{fig: EmbeddingInterosculation}. Therefore, $\phi_e^-(X_e)$ inter-osculates with the hyperplane $H'$ of $X_{\iota(e)}$. Similarly, if the vertical hyperplane $H$ of $X_e\times [0,1]$ inter-osculates with a hyperplane of $X_{\tau(e)}$, then the subcomplex $\phi_e^+(X_e)$ inter-osculates with a hyperplane of $X_{\tau(e)}$. This proves (3). 
\end{proof}
\end{lem}
\begin{figure}[h]
    \centering
    \includegraphics[width=0.75\linewidth]{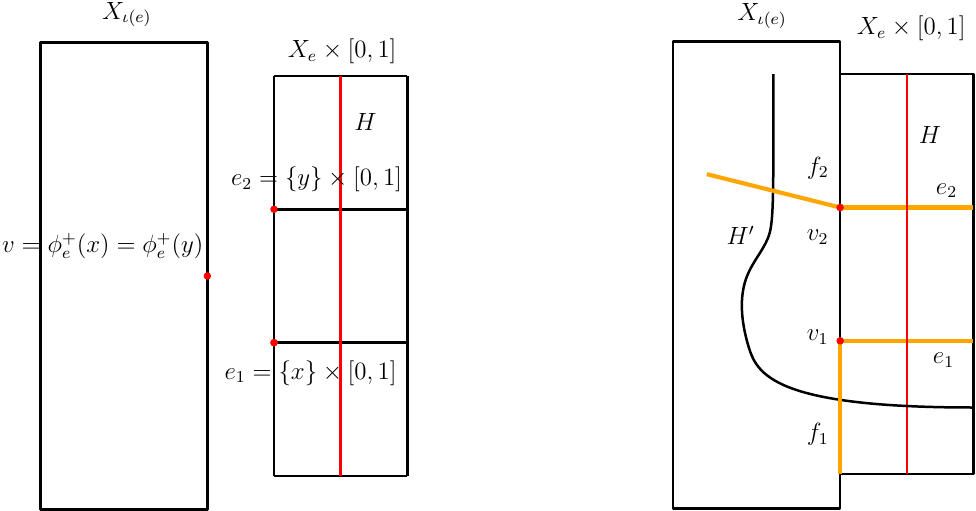}
    \caption{In the left figure, the red hyperplane $H$ directly self-osculates at $(v;e_1,e_2)$. In the right figure, the red hyperplane $H$ and the black hyperplane $H'$ cross at $(v_1;e_1,f_1)$ and osculate at $(v_2;e_2,f_2)$.}
    \label{fig: EmbeddingInterosculation}
\end{figure}

\section{Constant vertex spaces and groups}\label{section: constant}
\begin{defn}\label{def of constant}
A graph of nonpositively curved cube complexes $X_\Gamma$ is said to have  \emph{constant vertex spaces} if there exists a nonpositively curved cube complex $X_V$, called the \emph{constant vertex space}, together with a cubical isometry $\psi_u \colon X_u \to X_V$ for each vertex space $X_u$, such that $\psi_{\iota(e)} \circ \phi_e^- = \psi_{\tau(e)} \circ \phi_e^+$ for every $e$ of $E(\Gamma)$.

A graph of groups $G_\Gamma$ is said to have \emph{constant vertex groups} if there
exists a group $G_V$, called the \emph{constant vertex group}, together with an isomorphism
$\Psi_u: G_u \to G_V$ for each vertex group $G_u$, such that $\Psi_{\iota(e)} \circ \Phi_e^- = \Psi_{\tau(e)} \circ \Phi_e^+$ for every $e \in E(\Gamma)$.
\end{defn}

\begin{rmk}
A graph of nonpositively curved cube complexes $X_\Gamma$ has an associated graph of groups $G_\Gamma$ over $\Gamma$, whose vertex groups are $G_v = \pi_1 X_v$ and whose edge groups are $G_e = \pi_1 X_e$. The attaching maps $\phi_e^-, \phi_e^+$ induce monomorphisms of edge groups $\Phi_e^- = (\phi_e^-)_*: \pi_1X_e\to \pi_1X_{\iota(e)}$ and $\Phi_e^+ = (\Phi_e^+)_*: \pi_1X_e\to \pi_1X_{\tau(e)}$. In particular, if $X_\Gamma$ has constant vertex spaces, then $G_\Gamma$ has constant vertex groups.
\end{rmk}

\begin{lem}\label{retraction of group}
Let $G$ be the fundamental group of a graph of groups $G_\Gamma$ with constant vertex groups. Then for each vertex group $G_v$, there exists a retraction $\rho: G\to G_v$ whose restriction to each vertex group $G_u$ is an isomorphism $G_u\cong G_v$. 
\end{lem}
\begin{proof}
Let $\Psi_u:G_u\to G_V$ be the isomorphism from each vertex group $G_u$ to the constant vertex group $G_V$.
Recall that $G$ is generated by vertex groups $G_u$ and $t_e$ for $e\in E(\Gamma)$.
Consider the retraction $\rho: G\to G_v$ defined by
\[
\rho(g) = 
\begin{cases}
\Psi_v^{-1}\circ \Psi_u(g), & \text{if $g\in G_u$ for $u\in V(\Gamma)$}\\
1, & \text{if $g= t_e$ for $e\in E(\Gamma)$}.
\end{cases}
\]
Clearly, the restriction of $\rho$ to each vertex group $G_u$ is an isomorphism. Since $\Psi_{\iota(e)} \circ \Phi_e^- = \Psi_{\tau(e)} \circ \Phi_e^+$, for each $g\in G_e$, we have
\[
\rho(t_e\Phi_e^-(g)t_e^{-1})
= \rho(\Phi_e^-(g))
= \Psi_v^{-1}\circ \Psi_{\iota(e)}\circ \Phi_e^-(g)
= \Psi_v^{-1}\circ \Psi_{\tau(e)}\circ \Phi_e^+(g)
= \rho(\Phi_e^+(g))
\]
and so $\rho$ is well-defined.
\end{proof}
\begin{lem}\label{retraction of space}
Let $X$ be a cube complex that splits as a graph of nonpositively curved cube complexes $X_\Gamma$ with constant vertex spaces. Then for each vertex space $X_v$, there exists a retraction $r: X\to X_v$ such that the following hold:
\begin{enumerate}
    \item The restriction of $r$ to each vertex space $X_u$ is a cubical isometry $X_u\cong X_v$.
    \item For any two parallel directed edges $\dr{a}\para \dr{b}$ of $X$, if $r(\dr{a})$ is an edge of $X_v$, then $r(\dr{a})\para_{X_v}r(\dr{b})$.
\end{enumerate}
\end{lem}
\begin{proof}
Let $\psi_u:X_u\to X_V$ be the cubical isometry from each vertex space $X_u$ to the constant vertex space $X_V$. Consider the retraction $r\colon X \to X_v$ defined as follows. On a vertex space $X_u$, set $r|_{X_u}=\psi_v^{-1}\circ \psi_u$.
On a thickened edge space $X_e\times [0,1]$, define
\[
r(a,t)=r(\phi_e^-(a))
=\psi_v^{-1}\circ \psi_{\iota(e)}\circ \phi_e^-(a),
\]
where $a\in X_e$ and $t\in [0,1]$.
Because $\psi_{\iota(e)} \circ \phi_e^- = \psi_{\tau(e)} \circ \phi_e^+$, the retraction $r$ is well-defined.
By our construction, the restriction of~$r$ to each vertex space $X_u$ is a cubical isometry $X_u\cong X_V\cong X_v$, and so (1) holds.

To prove (2), let $\dr{a} = \dr{e_0}\para \dr{e_1}\para \cdots \para \dr{e_n} = \dr{b}$ be a sequence of directed edges where~$\dr{e_{i-1}}, \dr{e_i}$ are elementary parallel (i.e. they appear as opposite sides of a square) for each~$i$. 
Note that if $\dr{a} = \{p\}\times [0,1]$ is a horizontal edge of $X_e\times [0,1]$, then $r(\dr{a}) = r(\phi_e^-(p))$ is a point. Thus, if $r(\dr{a})$ is an edge of $X_v$, then each $\dr{e_i}$ is dual to a hyperplane that is not the vertical hyperplane of any thickened edge space. In particular, each $\dr{e_i}$ is contained in a vertex space.

For each $i$, if $\dr{e_{i-1}}, \dr{e_i}$ are contained in the same vertex space $X_u$, then they are opposite sides of a square of $X_u$. By (1), $r$ induces a cubical isometry $X_u\cong X_v$, which sends the square to a square of $X_v$, and sends $\dr{e_{i-1}}, \dr{e_i}$ to elementary parallel edges $r(\dr{e_{i-1}})\para_{X_v} r(\dr{e_i})$ in $X_v$. 
If $\dr{e_{i-1}}, \dr{e_i}$ are contained in adjacent vertex spaces, say $\dr{e_{i-1}}\subset X_{\iota(e)}$ and $\dr{e_i}\subset X_{\tau(e)}$, then the square containing them has sides 
\[
e_{i-1} = \phi_e^-(f),\quad y\times [0,1],\quad e_i = \phi_e^+(f), \quad x\times [0,1]
\]
where $f$ is an edge of $X_e$ with endpoints $x,y$, and $\dr{e_{i-1}}, \dr{e_i}$ are either both directed from $x$ to $y$ or both from $y$ to $x$. Then 
\[
r(\dr{e_{i-1}})
= \psi_v^{-1}\circ \psi_{\iota(e)}\circ \phi_e^-(\dr{f})
= \psi_v^{-1}\circ \psi_{\tau(e)}\circ \phi_e^+(\dr{f})
= r(\dr{e_{i}}). 
\]
In particular, $r(\dr{e_{i-1}})\para_{X_v} r(\dr{e_i})$.
Thus, in both cases, we obtain parallel edges $r(\dr{e_{i-1}})\para_{X_v} r(\dr{e_i})$ of $X_v$ for each $i$. Therefore, the retraction $r$ sends the sequence of parallel edges $\dr{a} = \dr{e_0}\para \cdot \para \dr{e_n} = \dr{b}$ to a sequence of parallel edges $r(\dr{a}) = r(\dr{e_0})\para_{X_v} \cdots \para_{X_v} = r(\dr{e_n}) = r(\dr{b})$ of $X_v$.
Therefore, $\dr{a}$ and $\dr{b}$ have parallel images $r(\dr{a})\para_{X_v}r(\dr{b})$ in $X_v$.
\end{proof}

\begin{cor}\label{cor of special}
Let $X$ be a cube complex that splits as a graph of special cube complexes $X_\Gamma$ with constant vertex spaces. Then $X$ is special provided that the attaching maps of edge spaces satisfy the following:
\begin{enumerate}[label=(\arabic*)]
    \item the attaching map $X_e\to X_{\iota(e)}$ is an embedding for each $e\in E(\Gamma)$;
    \item \label{attach no inter-osculation}
    the image of the attaching map $X_e\to X_{\iota(e)}$ does not inter-osculate with hyperplanes of $X_{\iota(e)}$ for each $e\in E(\Gamma)$.
\end{enumerate}
\end{cor}
\begin{proof}
We show that the four hyperplane pathologies in Definition \ref{def of special} cannot happen in~$X$. Choose a base vertex $v$ of $\Gamma$. Then Lemma \ref{retraction of space} yields a retraction $r: X\to X_v$. Because $X_v$ is special, the four hyperplane pathologies cannot happen in $X_v$. We show that each pathology in $X$ leads to a pathology in $X_v$ via the retraction map $r$.

\emph{Self-corssing:}
Suppose that a hyperplane $H$ of $X$ crosses itself at $(x; e_1, e_2)$ for parallel edges $e_1\para e_2$. By Lemma \ref{lemma of vertical hyperplane}, $H$ cannot be the vertical hyperplane of any thickened edge space $X_e \times [0,1]$. Thus, $e_1$ and $e_2$ are not horizontal edges of any $X_e \times [0,1]$, and hence belong to some vertex space $X_u$. Therefore, $r(e_1), r(e_2)$ are edges of the vertex space $X_v$. By Lemma \ref{retraction of space}, the restriction of $r$ to $X_u$ is an isometry $X_u \cong X_v$, sending $e_1\para e_2$ to parallel edges $r(e_1) \parallel_{X_v} r(e_2)$. Thus, the hyperplane of $X_v$ dual to $r(e_i)$ crosses itself at $(r(x); r(e_1), r(e_2))$.

\emph{$2$-sided:}
If a hyperplane $H$ of $X$ is $2$-sided, then there exists a directed edge $\dr{e}$ dual to $H$ such that $\dr{e} \parallel \dl{e}$. By Lemma \ref{lemma of vertical hyperplane}, $H$ is not a horizontal hyperplane of any thickened edge space. Thus, $e$ is not a horizontal edge of any $X_e \times [0,1]$, and so belongs to a vertex space $X_u$. Then the isometry $r|_{X_u}$ sends $\dr{e}$ to an edge $r(\dr{e})$ of $X_v$. By Lemma \ref{retraction of space}, the retraction $r$ sends $\dr{e} \parallel \dl{e}$ to parallel directed edges $r(\dr{e}) \para_{X_v} r(\dl{e})$ of $X_v$. In particular, the hyperplane of $X_v$ dual to $r(\dr{e})$ is $1$-sided.

\emph{Direct self-osculation:}
Because each $\phi_e^-: X_e\to X_{\iota(e)}$ is an embedding, the composition $\psi_{\iota(e)}\circ \phi_e^- = \psi_{\tau(e)}\circ \phi_e^+$ is an embedding for each $e$. In particular, each attaching map $\phi_e^+: X_e\to X_{\tau(e)}$ is also an embedding. Thus, all attaching maps are embeddings. 
Suppose that a hyperplane $H$ of $X$ directly self-osculates at $(x;\dr{e_1}, \dr{e_2})$, then $\dr{e_1}\para \dr{e_2}$. Because the attaching maps are embeddings, by Lemma \ref{lemma of vertical hyperplane}, $H$ cannot be a horizontal hyperplane of any thickened edge space. So $e_1,e_2$ belong to a vertex space $X_v$. The isometry $r|_{X_v}$ then sends $\dr{e_1}, \dr{e_2}$ to parallel edges $r(\dr{e_1})\para_{X_v}r(\dr{e_2})$ of $X_v$. Therefore, the hyperplane of $X_v$ dual to $r(\dr{e_1})$ directly self-osculates at $(r(x); r(\dr{e_1}), r(\dr{e_2}))$.
\begin{figure}
    \centering
    \includegraphics[width=0.75\linewidth]{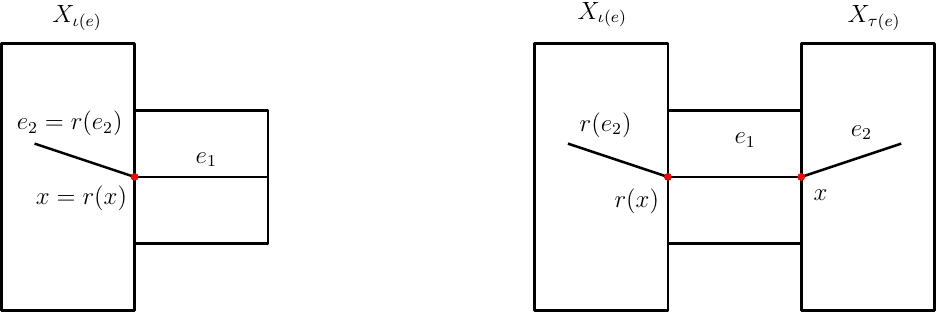}
    \caption{The left and right figures show the cases in which $e_2$ and $x$ belong to $X_{\iota(e)}$ and $X_{\tau(e)}$ respectively.}
    \label{fig: flip}
\end{figure}

\emph{Inter-osculation:}
If a two hyperplanes $H_1, H_2$ inter-osculate, say they cross at $(x;e_1,e_2)$ and osculate at $(y;f_1,f_2)$, where $e_i, f_i$ are dual to $H_i$. 
When neither of $H_1,H_2$ is a vertical hyperplane of any $X_e\times [0,1]$, then edges $e_1,e_2,f_1,f_2$ are not horizontal edges of any thickened edge space $X_e\times [0,1]$. Thus, they belong to edges of some vertex spaces, and their images $r(e_1), r(e_2), r(f_1), r(f_2)$ are edges of $X_v$. By Lemma \ref{retraction of space}, $r(e_1)\para_{X_v} r(f_1)$ and $r(e_2)\para_{X_v} r(f_2)$. In particular, the hyperplanes of $X_v$ dual to $r(e_1)$ and $r(e_2)$ cross at $(r(x); r(e_1), r(e_2))$ and osculate at $(r(y); r(f_1), r(f_2))$, and hence inter-osculate.

Now, assume that at least one of $H_1, H_2$ is a vertical hyperplane of some thickened edge space $X_e\times [0,1]$. Because vertical hyperplanes in distinct thickened edge spaces do not cross, $H_1, H_2$ cannot both be vertical hyperplanes. Assume that $H_1$ is the vertical hyperplane of $X_e\times [0,1]$ and $H_2$ is not a vertical hyperplane, then the two edges $e_2,f_2$ belong to some vertex spaces (they may belong to distinct vertex spaces). By Lemma~\ref{retraction of space}, the retraction $r: X\to X_{\iota(e)}$ sends parallel edges $e_2\para f_2$ to parallel edges $r(e_2)\para_{X_{\iota(e)}} r(f_2)$ of $X_{\iota(e)}$. 
Furthermore, recall that $e_1,e_2$ are incident at $x$. Suppose $e_1 = \{p\}\times [0,1]$ for some vertex $p\in V(X_e)$. If $e_2\subset X_{\iota(e)}$, then $x = \phi_e^-(e_1)\in X_{\iota(e)}$. In this case, $e_2 = r(e_2)$ and $e_1$ are incident at $x = r(x)$. If $e_2\subset X_{\tau(e)}$, then $x = \phi_e^+(e_2)\in X_{\tau(e)}$. In particular, $r(x)\subset r(e_2)$ and 
\[
r(x) 
= \psi_{\iota(e)}^{-1}\circ \psi_{\tau(e)} \circ \phi_e^+(p)
= \psi_{\iota(e)}^{-1}\circ \psi_{\iota(e)} \circ \phi_e^-(p)
= \phi_e^-(p)\in e_1.
\]
Thus, in both cases, $r(e_2)$ and $e_1$ are incident at the vertex $r(x)$ as shown in Figure \ref{fig: flip}. Similarly, $r(f_2)$ and $f_1$ are incident at the vertex $r(y)$. Therefore, the hyperplane of $X_{\iota(e)}$ dual to $r(e_2)$ inter-osculate with the vertical hyperplane $H$ of $X_e\times [0,1]$. 
However, condition \ref{attach no inter-osculation} implies that the image of the attaching map $X_e\to X_{\iota(e)}$ does not inter-osculate with hyperplanes of $X_{\iota(e)}$.
By Lemma \ref{lemma of vertical hyperplane}, $H$ therefore does not inter-osculate with any hyperplane of $X_{\iota(e)}$, a contradiction. 

Therefore, each hyperplane pathology of $X$ leads to either the same hyperplane pathology in a vertex space $X_v$ or a contradiction of condition \ref{attach no inter-osculation}. Since the vertex space $X_v$ is special, this completes the proof.
\end{proof}

\begin{rmk}
The assumption of constant vertex groups can be relaxed under stronger hypotheses on the attaching maps. For example, see \cite[Theorem 5.2]{HaglundWise2010coxeter}.
\end{rmk}

\section{Locally constant vertex spaces and groups}\label{section: locally constant}
\begin{defn}\label{def of locally constant}
A graph of nonpositively curved cube complexes $X_\Gamma$ is said to have \emph{locally constant vertex spaces} if there exists a cubical isometry $\theta_e: X_{\iota(e)}\to X_{\tau(e)}$ for each edge $e\in E(\Gamma)$ such that $\theta_e\circ \phi_e^- = \phi_e^+$.

A graph of groups $G_\Gamma$ is said to have \emph{locally constant vertex groups} if there exists an isomorphism $\Theta_e: G_{\iota(e)}\to G_{\tau(e)}$ for each edge $e\in E(\Gamma)$ such that $\Theta_e\circ \Phi_e^- = \Phi_e^+$.
\end{defn}

\begin{rmk}
Let $X_\Gamma$ be a graph of nonpositively curved cube complexes. If $X_\Gamma$ has constant vertex spaces, then it also has locally constant vertex spaces with the isometry $\theta_e = \psi_{\tau(e)}^{-1}\circ \psi_{\iota(e)}: X_{\iota(e)}\to X_{\tau(e)}$ for each edge $e\in E(\Gamma)$. Similarly, a graph of groups with constant vertex groups also has locally constant vertex groups. 
\end{rmk}

Let $X_\Gamma$ be a graph of nonpositively curved cube complexes with locally constant vertex spaces. Fix a base vertex $v\in V(\Gamma)$. Suppose that $\gamma: [0,n]\to \Gamma$ is a loop of $\Gamma$ based at $\gamma(0) = \gamma(n) = v$, which passes through edges $e_1 = \gamma[0,1],\dots, e_n = \gamma[n-1,n]$ in order. Since each $e_i$ gives an isometry $\phi_{e_i}: X_{\iota(e_i)}\to X_{\tau(e_i)}$, the loop $\gamma$ yields an a composition of isometries 
\[
\theta_{\gamma}: X_v = X_{\gamma(0)}\to X_{\gamma(1)}\to \cdots \to X_{\gamma(n)} = X_v,
\]
where each map in the composition is some $\theta_{e_i}$ or its inverse. Note that the isometry depends only on the homotopy class of $\gamma$. Hence, this defines a homomorphism
\[
\theta: \pi_1(\Gamma,v)\to \Aut(X_v),
\]
where $\Aut(X_v)$ is the set of combinatorial automorphisms of $X_v$. 

Similarly, let $G_\Gamma$ be a graph of groups with locally constant vertex groups, and fix a base vertex $v\in V(\Gamma)$. Each loop $\gamma$ of $\Gamma$ based at $v$ gives an automorphism $\Theta_\gamma: G_v\to G_v$, which depends only on the homotopy class of $\gamma$. This defines a homomorphism 
\[
\Theta: \pi_1(\Gamma,v)\to \Aut(G_v). 
\]

\begin{defn}
Let $X_\Gamma$ be a graph of nonpositively curved cube complexes with locally constant vertex spaces, and let $G_\Gamma$ be a graph of groups with locally constant vertex groups. Fix a base vertex $v\in V(\Gamma)$.
The \emph{monodromy group} of $X_\Gamma$ at $v$ is the image the homomorphism $\theta: \pi_1(\Gamma,v) \to \Aut(X_v)$. 
The \emph{monodromy group} of $G_\Gamma$ at $v$ is the image of the homomorphism $\Theta: \pi_1(\Gamma,v) \to \Aut(G_v)$.
\end{defn}
\begin{rmk}
For different choices of base vertices, the resulting monodromy groups are well defined up to conjugacy. Thus, we may speak of the monodromy group without specifying a base vertex.
\end{rmk}

\begin{lem}\label{lemma: local become global}
Given a cube complex $X$ that splits as a graph of nonpositively curved cube complexes $X_\Gamma$ with locally constant vertex spaces, and a group $G$ that splits as a graph of groups $G_\Gamma$ with locally constant vertex groups.
\begin{enumerate}
    \item If $G_\Gamma$ (resp. $X_\Gamma$) has trivial monodromy group, then $G_\Gamma$ has constant vertex groups (resp. $X_\Gamma$ has constant vertex spaces).
    \item If $G_\Gamma$ has finite monodromy group, then $G$ has a finite index subgroup $\hG$ which splits as a graph of groups with constant vertex groups.
    \item If $X_\Gamma$ has compact vertex spaces, then $X$ has a finite sheeted cover $\hX$ which splits as a graph of nonpositively curved cube complexes with constant vertex spaces.
\end{enumerate}
\end{lem}
\begin{proof}
When the monodromy group of $G_\Gamma$ is trivial. Choose a vertex $v$ of $\Gamma$ and take the constant vertex group to be $G_V = G_v$. For each vertex group $G_u$, there exists a path $\gamma_{u,v}: [0,n]\to \Gamma$ from $u$ to $v$ passing through edges $e_1,\cdots, e_n$. This gives a composition of isomorphisms
\[
\Psi_u := \Theta_{\gamma_{u,v}}: G_{u} = G_{\gamma(0)}\to G_{\gamma(1)}\to \cdots \to G_{\gamma(n)} = G_v,
\]
where each isomorphism $G_{\gamma(i-1)}\to G_{\gamma(i)}$ is either $\Theta_{e_i}$ or $\Theta_{e_i}^{-1}$. Because $\Theta: \pi_1(\Gamma,v)\to \Aut(G_v)$ has trivial image, the isomorphism $\Psi_u$ is independent of the choice of path. Indeed, for any other path $\xi_{u,v}$ from $u$ to $v$ and the associated isomorphism $\xi_{u,v}: G_u\to G_v$, let $\overline{\xi}_{u,v}$ be the opposite path from $v$ to $u$. Then the concatenation $\overline{\xi}_{u,v}*\gamma_{u,v}$ is a loop based at $v$, and so the composition 
\[
\Theta_{\overline{\xi}_{u,v}*\gamma_{u,v}} 
= \Theta_{\gamma_{u,v}}\circ \Theta_{\overline{\xi}_{u,v}}
= \Theta_{\gamma_{u,v}} \circ \Theta_{\xi_{u,v}}^{-1}
\]
is the identity element in $\Aut(G_v)$. Thus, $\Theta_{\xi_{u,v}} = \Theta_{\gamma_{u,v}}$.

Now, for each edge $e$ of $\Gamma$, the monomorphisms $\Phi_e^-: G_e\to G_{\iota(e)}$ and $\Phi_e^+: G_e\to G_{\tau(e)}$ satisfy $\Theta_e\circ\Phi_e^- = \Phi_e^+$. If $\gamma$ is a path from $\tau(e)$ to $v$, then $e*\gamma$ is a path from $\iota(e)$ to $v$. Thus, by our construction, we have
\[
\Psi_{\tau(e)} \circ \Phi_e^+
= \Theta_{\gamma} \circ \Phi_e^+
= \Theta_{\gamma} \circ \Theta_e\circ\Phi_e^-
= \Theta_{e*\gamma}\circ \Phi_e^-
= \Psi_{\iota(e)} \circ \Phi_e^-.
\]
This shows that $G_\Gamma$ has constant vertex groups. The argument for graph of nonpositively curved cube complexes $X_\Gamma$ is similar, with isomorphisms replaced by cubical isometries of vertex spaces. This proves (1).

For (2), let $T$ be a spanning tree of the underlying graph $\Gamma$. Then the set of edges of $\Gamma\setminus T$ forms a generating set for $\pi_1\Gamma$.
Consider the homomorphism $\pi: G\to \pi_1(\Gamma,v)$, which sends each vertex group to identity and each $t_e$ to the generator corresponding to~$e$. Consider the composition
\[
f = \Theta\circ \pi: G\to \pi_1(\Gamma,v)\to \Aut(G_v).
\]
Because the monodromy group $\Theta(\pi_1(\Gamma,v))$ is finite, the kernel $\hG = \ker(f)$ of $f$ has finite index in $G$. We claim that the induced splitting of $\hG$ as a graph of groups has constant vertex groups. 

Let $\cT$ be the Bass--Serre tree associated to $G_\Gamma$. Restricting the action of $G$ on $\cT$ to $\hG$ yields a splitting of $\hG$ as a graph of groups $\hG_{\hGamma}$ over the quotient graph $\hGamma=\hG\backslash \cT$. Let $p: \hGamma \to \Gamma = G\backslash\cT$ be the covering map. For each vertex $\hv$ of $\hGamma$, let $\tv$ be a lift of $\hv$ in $\cT$. Then the vertex group over $\hv$ is
\[
\hG_{\hv} = \stab_{\hG}(\tv) = \hG\cap \stab_G(\tv) = \hG\cap G_{p(\tv)}.
\]
Because $\pi$ sends the vertex group $G_{p(\tv)}$ to identity, $G_{p(\tv)}\subset \ker(\pi)\subset \ker(f) = \hG$. Thus, the vertex group over each vertex $\hv$ is $\hG_{\hv} = G_{p(\hv)}$. Similarly, the edge group over each edge $\he$ is $\hG_{\he} = G_{p(\he)}$, together with monomorphisms $\hPhi_{\he}^- = \Phi_e^-$ and $\hPhi_{\he}^+ = \Phi_e^+$.
As a result, for each edge $\he$ of $\hGamma$, the isomorphism $\Theta_{p(\he)}: G_{\iota(p(\he))}\to G_{\tau(p(\he))}$ gives an isomorphism
\[
\hTheta_{\he} = \Theta_{p(\he)}: \hG_{\iota(\he)} = G_{\iota(p(\he))}\to  G_{\tau(p(\he))}= \hG_{\tau(\he)},
\]
with $\hTheta_{\he}\circ \hPhi_{\he}^-
= \Theta_{p(\he)}\circ \Phi_{p(\he)}^-
= \Phi_{p(\he)}^+
= \hPhi_{\he}^+$.
This proves that the splitting $\hG_{\hGamma}$ of $\hG$ has locally constant vertex groups. 

Moreover, fix a vertex $\hv$ of $\hGamma$. For each loop $\hgamma$ based at $\hv$ that passes through edges $\he_1,\cdots, \he_n$, its projection $p(\hgamma)$ in $\Gamma$ is a loop based at $p(\hv)$ passing through edges $p(\he_1),\cdots, p(\he_n)$. Thus, the homomorphism $\Theta: \pi_1(\hGamma,\hv)\to \Aut(G_{\hv})$ factors through as 
\[
\hTheta = \Theta\circ p_*: \pi_1(\hGamma,\hv)\to \pi_1(\Gamma,v)\to \Aut(G_{p(\hv)}) = \Aut(\hG_{\hv}),
\]
where $p_*$ is induced by the covering map $p: \hGamma = \hG\backslash \cT\to \Gamma = G\backslash \cT$. We therefore obtain the following commutative diagram.
\[\begin{tikzcd}
	\hG & G & \\
	{\pi_1(\hGamma, \hv)} & {\pi_1(\Gamma,v)} & {\Aut(G_v)}
	\arrow[hook, from=1-1, to=1-2]
	\arrow["{\widehat{\pi}}"', two heads, from=1-1, to=2-1]
	\arrow["\pi", two heads, from=1-2, to=2-2]
	\arrow["f", from=1-2, to=2-3]
	\arrow["{p_*}"', from=2-1, to=2-2]
	\arrow["\Theta"', from=2-2, to=2-3]
\end{tikzcd}\]
Recall that $\hG = \ker(f)$.
Because $\widehat{\pi}$ and $\pi$ are surjective, we have 
\[
\ima(p_*)
= p_*\circ \widehat{\pi}(\hG)
= \pi(\hG) 
= \pi(\ker(f)) \subset  \ker(\Theta).
\]
In particular, $\hTheta = \Theta \circ p_*$ sends $\pi_1(\hGamma,\hv)$ to identity. Therefore, $\hG_{\hGamma}$ has trivial monodromy group, and hence has constant vertex groups by (1).

For (3), the graph of nonpositively curved cube complex $X_\Gamma$ yields a splitting of $\pi_1X$ as a graph of groups with locally constant vertex groups. Furthermore, the monodromy group lies in $\Aut(X_v)$ for a chosen base vertex $v$. When $X_v$ is compact, the automorphism group $\Aut(X_v)$ is finite, and so $X_\Gamma$ has finite monodromy group. Thus, by (2), $\pi_1X$ admits a finite index subgroup which splits as a graph of groups with constant vertex groups. The corresponding finite-sheeted cover $\hX$ of $X$ then admits a splitting with constant vertex spaces. 
\end{proof}

\section{Virtual specialness of graphs of cube complexes with locally constant vertex spaces}\label{section: virtually special}
In this section, we prove our main theorems. We first review the statement of Theorem~\ref{cubical thm}.
\mainthm*
\begin{proof}
By Lemma \ref{lemma: local become global}(3), it suffices to assume that $X_\Gamma$ has constant vertex spaces with constant vertex space $X_V$, together with cubical isometries $\psi_u: X_u\to X_V$ from each vertex space $X_u$.
Choose a base vertex $v\in V(\Gamma)$. By Lemma \ref{retraction of space}, there exists a retraction $r: X\to X_v$ whose restriction to each vertex space $X_u$ is an isometry.
\[
r|_{X_u} = \psi_{v}^{-1}\circ \psi_u: X_u\to X_v.
\]
For each edge space $X_e$, the attaching maps $\phi_e^-: X_e\to X_{\iota(e)}$ and $\phi_e^+: X_e\to X_{\tau(e)}$ give a local isometry
\[
f_e := \psi_{v^{-1}}\circ \psi_{\iota(e)}\circ \phi_e^- =\psi_v^{-1}\circ \psi_{\tau(e)}\circ \phi_e^+: X_e\to X_v.
\]
Because $X$ is compact, $\Gamma$ has finitely many edges $e_1,\cdots, e_n$, which give finitely many local isometries $X_{e_1},\cdots, X_{e_n}\to X_v$ of virtually special cube complexes. By Lemma \ref{lemma of interosculation}, there exists a compact regular special cover $\hX_v$ of $X_v$ such that all elevations of $X_{e_i}$ to $\hX_v$ are embedded and do not inter-osculate with hyperplanes of $X_v$. 

By Lemma \ref{retraction of group}, there exists a retraction $\rho: \pi_1X\to \pi_1X_v$, whose restriction to each vertex group $\pi_1X_u$ is induced by $r|_{X_u}$. Take the subgroup $\hG = \rho^{-1}(\pi_1\hX_v)$. Since $\pi_1\hX$ has finite index in $\pi_1X$, $\hG$ also has finite index in $\pi_1X$. Let $\hX$ be the corresponding finite-sheeted cover of $X$. 

We claim that $\hX$ is special. 
To see this, let $\hX_v\times_{X_v}X$ be the fiber product with respect to the covering $\pi: \hX_v\to X_v$ and the retraction $r: X\to X_v$. By Remark \ref{remark of elevation and subgroup}, $\hX_v\times_{X_v}X$ is the covering $\hX$ of $X$ corresponding to $\hG$. The splitting $X_\Gamma$ of $X$ then induces a splitting of $\hX$ as a graph of nonpositively curved cube complexes $\hX_{\hGamma}$. Since $r|_{X_u}$ is an isometry to $X_v$ for each vertex space $X_u$, the preimage of each vertex space $X_u$ is 
\[
\hX_v\times_{X_v} X_u 
= \{(x,y)\in \hX_v\times X_u \mid \pi(x) = r(y)\}
\cong \hX_v.
\]
In particular, each vertex spaces of $\hX_{\hGamma}$ is isometric to $\widehat X_v$. Similarly, the preimage of each thickened edge space $X_e\times [0,1]$ is 
\[
\hX_v\times_{X_v} (X_e\times [0,1]) = (\hX_v\times _{X_v} X_e)\times [0,1],
\]
where the attaching maps are elevations of $X_e\to X_v$ to $\hX_v$. Hence $\widehat X$ is a graph of special cube complexes with constant vertex space $\widehat X_v$. 

By the choice of $\widehat X_v$, all these elevated attaching maps are embeddings, and their images do not inter-osculate with hyperplanes of the corresponding vertex spaces. Therefore, by Corollary \ref{cor of special}, the cube complex $\widehat X$ is special and $X$ is virtually special.
\end{proof}

\begin{rmk}
In general, if we do not assume locally constant vertex spaces, the virtual specialness of $X$ may fail even if the splitting has isometric vertex spaces. For example, in \cite[Part \RNum{2}, Section 2]{Wise1996NPCSquareComplexes} (see also \cite[Section 4]{Wise2007CompleteSquareComplexes}), Wise constructed a compact nonpositively curved square complex $X$, which splits as a graph of graphs whose underlying graph is a loop. In particular, the splitting has a single vertex space. However, since $X$ is not virtually clean, it is not virtually special. In fact, the fundamental group $\pi_1X$ is not even virtually special since it is not residually finite \cite[Corollary 1.4]{CapraceWesolek2018Indicability}.
\end{rmk}

\begin{cor}\label{cor of Gromov-hyperbolic}
Let $G$ be the fundamental group of a finite graph of groups with locally constant vertex groups $G_\Gamma$. Suppose that the monodromy group of $G_\Gamma$ is finite.
If the vertex groups are virtually compact special Gromov-hyperbolic, and the edge groups are quasiconvex in adjacent vertex groups, then $G$ is virtually compact special.
\end{cor}
\begin{proof}
By Lemma \ref{lemma: local become global}(2), it suffices to assume that $G_\Gamma$ has constant vertex groups.
Let $G_v$ be a vertex group, which is virtually compact special Gromov-hyperbolic by assumption. By Lemma \ref{retraction of group}, there exists a retraction $\rho: G\to G_v$. 
Let $\hV$ be a finite index subgroup of $G_v$ that is compact special. Then there exists a compact special cube complex $X_v$ with $\hV = \pi_1X_v$. Take the finite index subgroup $\hG = \rho^{-1}(\hV)$. 

We show that $\hG$ is compact special. Denote by $G_\Gamma$ the splitting of $G$ as a graph of groups, with the associated Bass--Serre tree $\cT$. Then it induces a splitting of $\hG$ over the underlying graph $\hGamma = \hG\backslash\cT$. Denote by $p: \hGamma = \hG\backslash \cT\to \Gamma = G\backslash \cT$ the projection. Then the vertex group over each $\hu\in V(\Gamma)$ is
\[
\hG_{\hu} = G_{p(\hu)}\cap \hG = G_{p(\hu)}\cap \rho^{-1}(\hV),
\]
and the edge group over each $\he\in E(\Gamma)$ is 
\[
\hG_{\he} = G_{p(\he)}\cap \hG = G_{p(\he)}\cap \rho^{-1}(\hV).
\]
The restriction of $\Phi_{p(\he)}^-, \Phi_{p(\he)}^+$ to $\hG_{\he}$ give monomorphisms
\[
\hPhi_{\he}^- = \Phi_{p(\he)}^-|_{\hG_{\he}}: \hG_{\he}\to \hG_{\iota(\he)},\quad
\hPhi_{\he}^+ = \Phi_{p(\he)}^+|_{\hG_{\he}}: \hG_{\he}\to \hG_{\tau(\he)}.
\]
Because the restriction of $\rho$ to each vertex group $G_u$ is an isomorphism $G_u\cong G_v$, its restriction to the subgroup $\hG_{\hu} = G_{p(\hu)}\cap \rho^{-1}(\hV)$ is an isomorphism $\hPsi_{\hu}: \hG_{\hu}\to \hV$ with $\hPsi_{\iota(\he)}\circ \hPhi_{\he}^- = \hPsi_{\tau(\he)}\circ \hPhi_{\he}^+$. This shows that the splitting $\hG_{\hGamma}$ of $\hG$ is a graph of groups with constant vertex groups.

Finally, since each edge group $G_e$ is quasiconvex in $G_{\iota(e)}, G_{\tau(e)}$, the finite index subgroup $\hG_{\he}\leq G_{p(\he)}$ is quasiconvex in $\hG_{\iota(\he)}, \hG_{\tau(\he)}$. 
Since $\hV=\pi_1\hX_v$, where $\hX_v$ is a compact special cube complex, for each edge group $\hG_{\he}$ there exists a compact cube complex $\hX_{\he}$ and a local isometry $f_e:\hX_{\he}\to \hX_v$ such that $(f_e)_*$ maps $\pi_1\hX_{\he}$ isomorphically onto $\hG_{\he}$ \cite[Proposition 7.2]{haglund-wise2012}.
Thus, the graph of groups $\hG_{\hGamma}$ gives rise to a graph of special cube complexes $\hX_{\hGamma}$ with the constant vertex space $\hX_v$ and edge spaces $\hX_{\he}$. Let $\hX$ be the total space of $\hX_{\hGamma}$, then $\hG = \pi_1\hX$. 
By our construction, $\hX_{\hGamma}$ has constant vertex spaces. Hence, by Theorem \ref{cubical thm}, $\pi_1\hX$ is virtually compact special. Since $\hG = \pi_1\hX$ has finite index in $G$, it follows that $G$ is virtually compact special.
\end{proof}

\begin{proof}[Proof of Theorem \ref{general result}]
By Lemma \ref{lemma: local become global}, a graph of groups $G_\Gamma$ with constant vertex groups has locally constant vertex groups with trivial monodromy group. The statement then follows from Corollary \ref{cor of Gromov-hyperbolic}.
\end{proof}
\begin{proof}[Proof of Theorem \ref{special of double}]
The statement follows from Theorem \ref{general result} since the double $G*_HG$ is the fundamental group of a graph of groups with constant vertex groups, whose underlying graph consists of a single edge joining two vertices.
\end{proof}

\section{Remarks on finite stature}\label{section: finite stature}
The following notion was introduced in \cite{HuangWise2019Stature}.
\begin{defn}
Let $G$ be a group and let $\{H_\lambda\}_{\lambda\in \Lambda}$ be a collection of subgroups of $G$. Then $(G,\{H_\lambda\}_{\lambda\in \Lambda})$ has \emph{finite stature} if for each $\mu\in \Lambda$, there are finitely many $H_\mu$-conjugacy classes of infinite subgroups of the form $H_\mu\cap C$, where $C$ is an intersection of $G$-conjugates of elements of $\{H_\lambda\}_{\lambda\in \lambda}$.
\end{defn}

The relation between finite stature and virtual specialness is given by the following:
\begin{thm}[{\cite[Theorem 1.4]{Huang--Wise2024fintiestature}}]\label{relation thm}
Let $X$ be a compact cube complex that splits as a graph of nonpositively curved cube complexes. Suppose each vertex group is Gromov-hyperbolic. Then the following are equivalent.
\begin{enumerate}
    \item $\pi_1X$ has finite stature with respect to the vertex groups.
    \item $X$ is virtually special.
\end{enumerate}
\end{thm}
Furthermore, the direction $(2) \Longrightarrow (1)$  follows from \cite[Theorem 6.12]{Huang--Wise2024fintiestature}, which does not reply on the hyperbolicity assumption. Thus, our results of virtual specialness in Theorem \ref{cubical thm} and Corollary \ref{cor of Gromov-hyperbolic} imply the following:
\begin{cor}\label{finite stature 1}
Let $X$ be a compact cube complex that splits as a graph of virtually special cube complexes $X_\Gamma$ with locally constant vertex spaces. Then $\pi_1X$ has finite stature with respect to the vertex groups.
\end{cor}
\begin{cor}\label{finite stature 2}
Let $G$ be the fundamental group of a finite graph of groups with locally constant vertex groups $G_\Gamma$. Suppose that the monodromy group of $G_\Gamma$ is finite.
If the vertex groups are virtually compact special Gromov-hyperbolic, and the edge groups are quasiconvex in adjacent vertex groups, then $G$ has finite stature with respect to the vertex groups.
\end{cor}

In fact, Corollaries~\ref{finite stature 1} and~\ref{finite stature 2} can be proved directly, without using virtual specialness. Together with Theorem~\ref{relation thm}, this gives an alternative proof of Corollary~\ref{cor of Gromov-hyperbolic}. However, this approach does not prove Theorem~\ref{cubical thm}, because Theorem~\ref{relation thm} assumes that the vertex groups are Gromov-hyperbolic.
\bibliographystyle{alpha}	
\bibliography{ref}

\end{document}